\documentclass{amsart}

\usepackage[OT2,T1]{fontenc}
\usepackage{tilmansdef}
\usepackage{tikz}
\usepackage[colorlinks=true,linkcolor=blue,citecolor=blue]{hyperref} 
\usepackage{tensor,enumitem,}

\usetikzlibrary{matrix,arrows,decorations,cd}

\usepackage{mathtools,stmaryrd}

\DeclareFontFamily{OT1}{pzc}{}
\DeclareFontShape{OT1}{pzc}{m}{it}{<-> s * [1.10] pzcmi7t}{}
\DeclareMathAlphabet{\mathpzc}{OT1}{pzc}{m}{it}

\DeclareMathOperator{\Pol}{Pol}
\DeclareMathOperator{\pol}{pol}

\newcommand{\FSch}[1]{\operatorname{FSch}_{#1}}

\newcommand{\DMod}[1]{\operatorname{Dmod}_{#1}}
\newcommand{\DModV}[1]{\operatorname{Dmod}_{#1}^{u}}
\newcommand{\DDModF}[1]{\mathbb{D}\operatorname{mod}_{#1}^{c}}

\newcommand{\DieuFor}{\mathbb D^{\operatorname{f}}}

\newcommand{\FGps}[1]{\operatorname{Fgps}_{#1}} 
\newcommand{\AbSch}[1]{\operatorname{AbSch}_{#1}} 
\newcommand{\Hopf}[1]{\operatorname{Hopf}_{#1}}

\newcommand{\fHopf}[1]{\widehat{\operatorname{Hopf}}_{#1}}

\renewcommand{\ell}{\mathpzc{l}}
\DeclareMathOperator{\frob}{frob}

\DeclareMathOperator{\hull}{hull}

\newcommand{\modtensor}[2]{\rtimes}
\newcommand{\modcotensor}[2]{\hom}

\newcommand{\Reg}[1]{\mathcal{O}_{#1}}

\DeclareMathOperator{\alg}{alg}
\DeclareMathOperator{\Cof}{Cof}

\DeclareMathOperator{\nil}{nil}

\DeclareMathOperator{\ungr}{ungr}

\newtheorem{notation*}[equation]{Notation}

\title{Graded $p$-polar rings and the homology of $\Omega^n\Sigma^nX$}
\author{Tilman Bauer}
\date\today
\thanks{The author would like to thank the Mittag-Leffler Institute for supporting this research.} 
\begin{document}

\begin{abstract}
As an extension of previous ungraded work, we define a graded $p$-polar ring to be an analog of a graded commutative ring where multiplication is only allowed on $p$-tuples (instead of pairs) of elements of equal degree. We show that the free affine $p$-adic group scheme functor, as well as the free formal group functor, defined on $k$-algebras for a perfect field $k$ of characteristic $p$, factors through $p$-polar $k$-algebras. It follows that the same is true for any affine $p$-adic or formal group functor, in particular for the functor of $p$-typical Witt vectors. As an application, we show that the homology of the free $E_n$-algebra $H^*(\Omega^n\Sigma^n X;\F_p)$, as a Hopf algebra, only depends on the $p$-polar structure of $H^*(X;\F_p)$ in a functorial way. 
\end{abstract}

\maketitle

\section{Introduction} \label{sec:intro}

In \cite{bauer:p-polar}, I introduced the notion of a $p$-polar $k$-algebra, which, roughly speaking, is a $k$-module with a $p$-ary associative and commutative multiplication defined on it. Here $p$ is a prime and $k$ is any commutative ring. If $k$ is a perfect field of characteristic $p$, I showed that the free affine abelian $p$-adic group functor on $\Spec R$ for a $k$-algebra $R$ factors through the category of $p$-polar $k$-algebras, and as a result, so does the functor of points for \emph{every} $p$-adic group defined over $k$.

In this sequel, I prove the corresponding results for graded commutative $k$-algebras, where $k$ is an (ungraded) perfect field of characteristic $p$. Both the definition of a graded $p$-polar $k$-algebra and the proofs are quite distinct, but not independent, from the ungraded case, and the results are more striking in the presence of a grading. This is demonstrated by applications concerning the mod-$p$ homology of $\Omega^n\Sigma^nX$.

In this paper, all graded objects are understood to be nonnegatively graded. 

\begin{defn}
Let $k$ be an ungraded ring and $\Alg_k$ the category of graded-commutative $k$-algebras. Let $M_k$ denote the category of graded $k$-modules $A$ together with graded symmetric $k$-multilinear maps $\mu\colon (A_j)^{\otimes_{k} p} \to A_{jp}$ for all $j$, and let $\pol_p\colon \Alg_k \to M_k$ denote the forgetful functor from graded commutative $k$-algebras to $M_k$, where $\mu$ is given by $p$-fold multiplication.

A \emph{graded $p$-polar $k$-algebra} is an object $A \in M_k$ which is a subobject of $\pol_p(B)$ for some algebra $B \in \Alg_k$.
\end{defn}

This definition agrees with the one given in \cite{bauer:p-polar} when $A$ is concentrated in degree $0$ (Lemma~\ref{lemma:olddef}). 

We denote the category of graded $p$-polar $k$-algebra by $\Pol_p(k)$. We also let $\alg_k$ and $\pol_p(k)$ denote the full subcategories of objects in $\Alg_k$ and $\Pol_p(k)$, respectively, that are of finite length as $k$-modules.

The restriction functors $\pol_p\colon \Alg_k \to \Pol_p(k)$ and $\pol_p \colon \alg_k \to \pol_p(k)$ are called \emph{polarization}. 

Our main algebraic results parallel those in \cite{bauer:p-polar}. Let $\AbSch{k}$ denote the category of representable, abelian-group-valued functors on $\Alg_k$, which is the opposite category of the category $\Hopf{k}$ of abelian Hopf algebras over $k$, and let $\AbSch{k}^p$ denote the full subcategory of functors taking values in abelian pro-$p$-groups. Let $\FGps{k}$ be the category of ind-representable functors on $\alg_k$ taking values in abelian groups, opposite to the category $\fHopf{k}$ of formal Hopf algebras, that is, cocommutative cogroup objects in $\Pro-\alg_k$. We will refer to objects $G$ of $\AbSch{k}$ and $\FGps{k}$ as affine and formal groups, respectively, and denote their associated (formal) Hopf algebras by $\Reg{G}$.

\begin{thm}\label{thm:freeschemefactorization}
Let $k$ be a perfect field of characteristic $p$. Then the forgetful functors $\AbSch{k}^p \to \Alg_k^{\op}$ resp. $\FGps{k} \to (\Pro-\alg_k)^{\op}$ have left adjoints $\Fr$, and these factor through $\pol_p$:
\[
\begin{tikzcd}
\Alg_k^{\op} \arrow[r,"\Fr"] \ar[dr,swap,"\pol_p"] & \AbSch{k}^p;\\
& \Pol_p(k)^{\op} \ar[u,"\tilde \Fr"]\\
\end{tikzcd}
\qquad
\begin{tikzcd}
\alg_k^{\op} \arrow[r,"\Fr"] \ar[dr,swap,"\pol_p"] & \FGps{k}\\
& \pol_p(k). \ar[u,"\tilde \Fr"]\\
\end{tikzcd}
\]
\end{thm}

\begin{corollary}\label{thm:affineschemefactorization}
Let $k$ be a perfect field of characteristic $p$ and $M \in \AbSch{k}^p$ or $M \in \FGps{k}$. Then $M$ factors naturally through $\pol$:
\[
\begin{tikzcd}
\Alg_k \arrow[r,"M"] \ar[dr,swap,"\pol_p"] & \{\text{abelian pro-$p$-groups}\};\\
& \Pol_p(k) \ar[u,"\tilde M"]\\
\end{tikzcd}
\qquad
\begin{tikzcd}
\alg_k \arrow[r,"M"] \ar[dr,swap,"\pol_p"] & \Ab\\
& \pol_p(k). \ar[u,"\tilde M"]\\
\end{tikzcd}
\]
\end{corollary}
\begin{proof}
Given any $M \in \AbSch{k}^p$ or $M \in \FGps{k}$, we have that
\[
M(R) = \Hom(\Spec R, M) = \Hom(\Fr(R),M),
\]
where the last Hom group is of objects of $\AbSch{k}^p$ or $\FGps{k}$, respectively. The Corollary now follows from Thm.~\ref{thm:freeschemefactorization}.
\end{proof}

In complete analogy with the ungraded case, we define:
\begin{defn}
An affine group $G$ over a field $k$ is called \emph{unipotent} if its representing Hopf algebra $\Reg{G}$ is unipotent (conilpotent). The latter means that the reduced comultiplication $\tilde{\Reg{G}} \to \tilde {\Reg{G}} \otimes \tilde {\Reg{G}}$ is nilpotent, where $\tilde {\Reg{G}}$ is the cokernel of the unit map.
The group $G$ is called \emph{of multiplicative type} if $\Reg{G}$ is isomorphic to a group algebra, possibly after passing to a finite field extension of $k$.

The Cartier dual $G^*$ of a formal group is the affine group represented by the Hopf algebra $\Hom^c(\Reg{G},k)$ of continuous $k$-linear maps. Cartier duality is an anti-equivalence between $\FGps{k}$ and $\AbSch{k}$.

A formal group $G$ is called \emph{connected} if is $G^*$ unipotent, and \emph{\'etale} if $G^*$ is of multiplicative type.
\end{defn}

The proof of Thm.~\ref{thm:freeschemefactorization} is easily reduced to the case of unipotent affine groups and connected formal groups, respectively. To prove that case, we use Dieudonné theory and develop a theory of Witt vectors for graded $p$-polar algebras. We proceed to summarize the main points.

There are Dieudonné equivalences (cf. Thm.~\ref{thm:generaldieudonne})
\[
D\colon (\AbSch{k}^u)^\op \to \DModV{k}
\]
and
\[
\DieuFor\colon (\FGps{k}^c)^\op\to \DDModF{k}
\]
with certain categories of modules over the ring $W(k)$ of $p$-typical Witt vectors over $k$ with certain Frobenius and Verschiebung operations. More generally, for a graded ring $R$, let $W_n(R)$ denote the ring of graded $p$-typical Witt vectors of length $n \leq \infty$ and
\[
CW^u(R) = \colim(W_1(R) \xrightarrow{V} W_2(R) \xrightarrow{V} \cdots)
\]
the graded abelian group of unipotent co-Witt vectors. Moreover, for a finite-dimensional graded $k$-algebra $R$ with nilradical $\nil(R)$, let
\[
CW^c(R)_i = \begin{cases}
CW^c(R_0); & i=0\\
CW^{c,u}(R)_i; & i \neq 0
\end{cases}
\]
be the graded abelian group of connected co-Witt vectors, where the ungraded $CW^c(R_0)$ is the classical group of connected co-Witt vectors \cite[\textsection 2.4]{fontaine:groupes-divisibles} and $CW^{c,u}(R) < CW^u(R)$ consists of those unipotent co-Witt vectors all of whose entries lie in $\nil(R)$.

\begin{lemma}\label{lemma:wittofppolar}
The functors $W_n$ and $CW^u$ from $\Alg_k$ to $\DModV{k}$ factor naturally through $\Pol_p(k)$, and the functor $CW^c\colon \alg_k \to \DDModF{k}$ factors naturally through $\pol_p(k)$.
\end{lemma}

We prove:
\begin{thm} \label{thm:dieudonneoffree}
For $R \in \Alg_k$, there are natural isomorphisms
\[
\DieuFor(\Fr^c(R)) \cong CW^c(R) \quad \text{for $R \in \alg_k$}
\]
and
\[
D(\Fr^u(R)) \cong CW^u(R) \quad \text{for $R \in \Alg_k$.}
\]
\end{thm}

Here, $\Fr^c$ and $\Fr^u$ denote the free connective formal and free unipotent affine group functors, respectively. This theorem together with Lemma~\ref{lemma:wittofppolar} are the main ingredients for proving Thm.~\ref{thm:freeschemefactorization} in the cases of unipotent affine and connected formal groups.

As an algebraic application, we show:
\begin{thm}\label{thm:lambdapcofree}
The affine group scheme of $p$-typical Witt vectors is the free unipotent abelian group scheme on the $p$-polar affine line, i.e. on the free $p$-polar algebra on a single generator. 
\end{thm}

This is an analog of the fact that the Hopf algebra representing the big Witt vectors, the algebra $\Lambda$ of symmetric functions, is cofree on a polynomial ring in one generator \cite{hazewinkel:cofree-coalgebras}. The corresponding Hopf algebra $\Lambda_p$ for $p$-typical Witt vectors is definitely not cofree on a (non-polar) algebra.

We conclude with a topological application. Let $X$ be a connected space. We show:

\begin{thm}\label{thm:omegansigman}
Let $p$ be a prime and $n \geq 1$. Then the functor $D_n\colon \Top \to \Hopf{\F_p}$ defined by $D_n(X) = H^*(\Omega^{n+1}\Sigma^{n+1} X;\F_p)$ factors through the forgetful functor $\Top \xrightarrow{H^*(-;\F_p)} \Alg_{\F_p} \xrightarrow{\pol} \Pol_p(\F_p)$.
\end{thm}

In particular, the Hopf algebra $H^*(\Omega^n\Sigma^n X;\F_p)$ depends only on the $p$-polar structure of $H^*(X;\F_p)$. 

We remark that under a finite-type condition, using work of Kuhn \cite{kuhn:quasi-shuffle} (and no $p$-polar algebras), we even have:

\begin{thm}\label{cor:kuhnthm}
For $n \geq 0$ and spaces $X$, $Y$ of finite type, $D_n(X) \cong D_n(Y)$ as Hopf algebras if and only if $H^*(X;\F_p) \cong H^*(Y;\F_p)$ as vector spaces with $p$th power maps. In particular, $D_n(X)$ only depends on the stable homotopy type of $X$, up to (noncanonical) isomorphism.
\end{thm}

\subsection{Outline of the paper}
In Section~\ref{sec:gradedppolar}, we study the formal properties of graded $p$-polar algebras, give examples, and prove a recognition criterion. In Section~\ref{sec:splittings}, we show that the proof of Thm.~\ref{thm:freeschemefactorization} splits into a number of distinct cases according to splittings of affine and formal schemes, respectively, and prove all cases except the case of unipotent affine and connected formal schemes. In Section~\ref{sec:witt}, we define Witt vectors for graded $p$-polar algebras as an extension of the usual $p$-typical Witt vectors of rings, and use those in Section~\ref{sec:dieudonne}, using graded Dieudonné theory, to prove the remaining cases of Thm.~\ref{thm:freeschemefactorization} along with Lemma~\ref{lemma:wittofppolar} and Thm.~\ref{thm:dieudonneoffree}. Section~\ref{sec:applications} treats algebraic properties and applications, including the proof of Thm.~\ref{thm:lambdapcofree}, and the final section~\ref{sec:iteratedloops} contains the topological applications to iterated loop spaces, including the proofs of Thms.~\ref{thm:omegansigman} and \ref{cor:kuhnthm}. 

\section{\texorpdfstring{Graded $p$-polar algebras}{Graded p-polar algebras}}\label{sec:gradedppolar}

We begin the study of $p$-polar $k$-algebras with some observations and examples.

A graded $p$-polar $k$-algebra $A$ does not supply a map $A^{\otimes_k p} \to A$ -- only elements of the same degree can be multiplied together. It will be useful to consider this stronger notion as well.

\begin{defn}
A \emph{strong graded $p$-polar $k$-algebra} is a graded $k$-module $A$ with a map $\mu\colon A^{\otimes_k p} \to A$ making the underlying ungraded $k$-module $A^{\ungr}$ into an ungraded $p$-polar algebra.
\end{defn}

If $A$ is merely a (weak) graded $p$-polar $k$-algebra, $A^{\ungr}$ cannot be given a $p$-polar $k$-algebra structure in general.

\begin{remark}
The embeddability $i\colon A \to \pol_p(B)$ of a graded $p$-polar algebra into a graded commutative algebra can be thought of as saying that for any elements $x_1,\dots,x_n \in A$ and a scalar $\lambda \in k$, there is at most one way of multiplying them together using $\mu$ (up to sign); namely, the element $\lambda i(x_1) \cdots i(x_n) \in B$, which is either in the image of $i$ or it isn't. 
\end{remark}

\begin{example}
The submodule $k\langle x^{p^i} \mid i \geq 0\rangle \subset k[x]$ is a sub-$p$-polar algebra of $\pol(k[x])$, where $|x|>0$. It is the free $p$-polar algebra on a generator $x$. This shows that in contrast to the ungraded case, even for $p=2$, $p$-polar algebras are much weaker structures than genuine algebras.
\end{example}

\begin{remark}
Let $\Mod_F$ denote the category of \emph{$F$-modules}, i.~e. graded $k$-modules $M$ with a linear map $F\colon M_{i} \to M_{pi}$ satisfying $F(\alpha x) = \alpha^p F(x)$ for $\alpha \in k$, $x \in M_{i}$. ($F=0$ if $p$ and $i$ are odd.) If $p \in k^\times$, the forgetful functor
\[
U_F\colon \Pol_p(k) \to \Mod_F; \quad M \to (M,F) \text{ with } F(x) = \mu(x,\dots,x)
\]
is an equivalence when restricted to evenly graded $M$ by what is classically called polarization. E.g., for $p=2$, $\mu(x,y) = \frac12 (F(x+y)-F(x)-F(y))$. However, the situation of most interest to us is when $k$ is a field of characteristic $p$. In that case, $U_F$ always has a left adjoint but is an equivalence only between the subcategories of $\Pol_p(k)$ and $\Mod_F$ whose underlying modules are at most $1$-dimensional in each degree.
\end{remark}

\begin{remark}\label{remark:ptypicalsplitting}
For any $p$-polar algebra $A$, the submodule $A_{(j)} = \bigoplus_{i \geq 0} A_{jp^i}$ is a polar subalgebra and direct factor, and thus
\[
A \cong A_0 \times \prod_{p \nmid j} A_{(j)}.
\]

We call a $p$-polar algebra of this form $A_{(j)}$ a \emph{$p$-typical} polar algebra of degree~$j$. The inclusion of such $p$-polar algebras into all $p$-polar algebras is biadjoint to the functor $A \mapsto A_{(j)}$. We write $\pol_{(j)}(A) := (\pol(A))_{(j)}$. This is a sub-$p$-polar algebra, but not a subalgebra of $A$.

In particular, if $p>2$, we see that every graded $p$-polar $k$-algebra splits as a product $A=A_{\odd} \times A_{\ev}$, where $A_{\odd} = \bigoplus_n A_{2n+1}$ and $A_{\ev} = \bigoplus_n A_{2n}$.
\end{remark}

\begin{example}
Consider the stable splitting $P=\Sigma^\infty (\CP^\infty)\hat{{}_p} \simeq P_1 \vee \cdots \vee P_{p-1}$ of the $p$-completion of complex projective space \cite{mcgibbon:rank-1-loops} with
\[
H^*(P_i) = \langle x^j \mid j \equiv i \pmod{p-1}\rangle < \F_p[x] = H^*(\CP^\infty,\F_p).
\]
By \cite{sullivan:genetics}, $P_{p-1}$ is the suspension spectrum of a space (the classifying space of the Sullivan sphere), but all other $P_i$ are not.
However, the maps
\[
P_i \hookrightarrow P = \Sigma^\infty (\CP^\infty)\hat{{}_p} \xrightarrow{\Sigma^\infty \Delta} \Sigma^\infty ((\CP^\infty)\hat{{}_p})^p \simeq P \wedge \cdots \wedge P \twoheadrightarrow P_i \wedge \cdots \wedge P_i
\]
induce (strong) graded $p$-polar algebra structures on $\tilde H^*(P_i)$, and the splitting $P\simeq P_1 \vee \cdots \vee P_{p-1}$ induces a splitting of $p$-polar algebras in cohomology. In fact,
\[
H^*(P_i;\F_p) \cong \bigoplus_{\substack{j \equiv i \;\;(p-1)\\p \nmid j}} \pol_{(j)} H^*(P;\F_p)
\]
So while the $P_i$ are not spaces for $i \neq p-1$, they do retain some likeness to spaces in that their cohomologies are $p$-polar algebras. This raises the question whether there is a reasonable notion of a ``$p$-polar space'' somewhere between connective spectra and spaces.
\end{example}

In a way, the definition of a $p$-polar $k$-algebra is wrong in the same way the definition of a manifold as a submanifold of $\mathbf R^n$ is wrong; it mentions an enveloping object which is not part of the data. The following proposition remedies this to a certain extent:

\begin{prop}\label{prop:hullembedding}
The functor $\pol\colon \Alg_k \to M_k$ has a left adjoint given for $A \in M_k$ by
\[
A \mapsto \hull(A) = \Sym(A)/(x_1\cdots x_p - \mu(x_1,\dots,x_p) \mid x_1,\dots,x_p \in A_i).
\]
An object $A \in M_k$ is a $p$-polar $k$-algebra iff the unit map of this adjunction, $u\colon A \to \pol(\hull(A))$, is injective.
\end{prop}
\begin{proof}
The existence and structure of the left adjoint, $\hull$, is obvious.

If $u$ is injective, $A$ is a $p$-polar algebra by definition. Conversely, if $A$ is $p$-polar, say $i\colon A \hookrightarrow \pol(B)$ for some $B \in \Alg_k$, then by the universal property of the left adjoint, there is a factorization
\[
\begin{tikzcd}
A \arrow[dr,hook,"i",swap] \arrow[r,"u"] & \pol(\hull(A)) \arrow[d,dashed]\\
& \pol(B).
\end{tikzcd}
\]
Since $i$ is injective, so is $u$.
\end{proof}

For the sake of explicitness, we will now give a list of axioms for objects of $M_k$ to be a $p$-polar algebra.

\begin{lemma}\label{lemma:olddef}
If $A=A_0$ (i.e. in the ungraded case), the definition of a $p$-polar $k$-algebra agrees with the one given in \cite{bauer:p-polar}; i.e., $A \in M_k$ is $p$-polar iff 
\begin{description}
	\item[(ASSOC)]  \label{lemma:olddef:homogeneousassoc}
For the symmetric group $\Sigma_{2p-1}$ permuting the elements $x_1,\dots,x_p$, $y_2,\dots,y_p \in A$,
\[
\mu(\mu(x_1,\dots,x_p),y_2,\dots,y_p)
\]
is $\Sigma_{2p-1}$-invariant.
\end{description}
\end{lemma}
\begin{proof}
Clearly, axiom (ASSOC) holds if $A$ is $p$-polar. Conversely, suppose that (ASSOC) holds, and let $i\colon A \to \hull(A)$ be the adjunction unit. In \cite{bauer:p-polar}, it was shown that (ASSOC) implies that for any $i \geq 0$ and any set of $1+i(p-1)$ elements $x_1,\dots,x_{1+i(p-1)}$, there is exactly one way of multiplying the $x_i$ together using $\mu$, and any other number of elements cannot be multiplied together. Write $\mu(x_1,\dots,x_{1+i(p-1)})$ for this unique product. Let $j\colon \bigoplus_{i=0}^\infty \Sym^{1+i(p-1)}(A) \hookrightarrow \Sym(A)$ be the inclusion and
\begin{align*}
\hull(A) \supseteq B =& \left(\bigoplus_{i=0}^\infty \Sym^{1+i(p-1)}(A)/j^{-1}((x_1\cdots x_p - \mu(x_1,\dots,x_p)))\right)\\
\cong & \left(\bigoplus_{i=0}^\infty \Sym^{1+i(p-1)}(A)\right)/\langle x_1\cdots x_{1+i(p-1)} - \mu(x_1,\dots,x_{1+i(p-1)})\rangle\\
\cong & \; A
\end{align*}
We have thus exhibited $A$ as a subobject of $\hull(A)$.
\end{proof}

Now we consider the general, graded case.

\begin{lemma}\label{lemma:nonperiodicchar}
Let $A$ be an object in $M_k$. Then $A$ is a $p$-polar $k$-algebra iff
\begin{enumerate}
	\item \label{lemma:nonperiodicchar:zero} $A_0$ is a $p$-polar $k$-algebra, and
	\item \label{lemma:nonperiodicchar:assoc} For the symmetric group $\Sigma_{2p}$ permuting the elements $x_1,\dots,x_{2p} \in A_j$, and elements $y_3,\dots,y_p \in A_{pj}$ (none if $p=2$),
	\[
	\mu(\mu(x_1,\dots,x_p),\mu(x_{p+1},\dots,x_{2p}),y_3,\dots,y_p)
	\]
is $\Sigma_{2p}$-invariant (up to multiplication with the sign of the permutation if $j$ is odd).
\end{enumerate}
\end{lemma}
\begin{proof}
The implication $A$ $p$-polar $\Rightarrow$ \eqref{lemma:nonperiodicchar:zero}, \eqref{lemma:nonperiodicchar:assoc} is straightforward. For the converse, we may assume by Remark~\ref{remark:ptypicalsplitting} without loss of generality that either $A=A_0$ or $A = A_{(j)}$ is $p$-typical. Lemma~\ref{lemma:olddef} takes care of the first case, so assume $A$ is $p$-typical and \eqref{lemma:nonperiodicchar:assoc} holds. We will show:

\medskip
\noindent\textbf{Claim:} $\Sym(M)_{(j)}$ is the free object on a $p$-typical $k$-module $M=M_{(j)}$ in the full subcategory of $M_k$ of objects satisfying \eqref{lemma:nonperiodicchar:assoc}.

\medskip
The result follows from the claim because of the commutative diagram
\[
\begin{tikzcd}
\Sym(A)_{(j)} \ar[r] \ar[d] & \Sym(A) \ar[d]\\
A \ar[r] & \hull(A)
\end{tikzcd}
\]
where the horizontal arrows are injections and the left-hand vertical arrow is given by dividing out by the $p$-polar ideal $(x_1\cdots x_p-\mu(x_1,\dots,x_p)) \cap \Sym(A)_{(j)}$.

\medskip

To prove the claim, let $M = M_{(j)}$ be a graded $p$-typical $k$-module. The free object $T_M$ in $M_k$ on $M$ is given recursively by
\[
(T_M)_n = M_n \oplus \Sym^p((T_M)_{\frac n p}),
\]
where $(T_M)_n = 0$ if $n \not\in \Z$. We call an element of $(T_M)_n$ a monomial if it is either an element of $M_n$ or a monomial $\{x_1,\dots,x_p\} \in \Sym^p$ on monomial elements in $x_i \in (T_M)_{\frac n p}$, using curly braces for equivalence classes of tensors in $\Sym^p$. Clearly, by linearity, any element of $T_M$ is a linear combination of monomials. One could describe these monomial elements as some kinds of labelled trees. While this is a good picture to have in mind, I will not use that language.

Define recursively an equivalence relation $\sim$ on monomials in $T_M$ (and hence, by linear extension, on all of $T_M$) generated by 
\begin{multline*}
\{\{x_1,\dots,x_p\},\{x_{p+1},\cdots,x_{2p}\},y_3,\cdots, y_p\}\\
 \sim (-1)^{|x_1|} \{\{x'_1,\dots,x'_{p-1},x'_{p+1}\},\{x'_{p},x'_{p+2},\cdots,x'_{2p}\},y'_3,\cdots, y'_p\}
\end{multline*}
iff $y_j \sim y'_j$ for $3 \leq h \leq p$, $x_i \sim x'_i$ for $1 \leq i \leq 2p$.
Then $T_M/\sim$ is the free object in $M_k$ satisfying \eqref{lemma:nonperiodicchar:assoc}. (The $\Sigma_{2p}$-invariance is equivalent to the invariance under interchanging $x_p$ and $x_{p+1}$, given the guaranteed $\Sigma_p \times \Sigma_p$-invariance.)

There is a linear map $f\colon T_M \to \Sym(M)_{(j)}$ given on monomials by $f(m) = m$ for $m \in M$ and $f(\{x_1,\dots,x_p\}) = f(x_1)\cdots f(x_p)$ for $\{x_1,\dots, x_p\} \in \Sym^p(T_M)$. We obtain an induced map $\overline f\colon T_M/\sim \to \Sym(M)_{(j)}$ which we claim to be an isomorphism.

To show surjectivity of $\overline f$, let $X=x_1,\dots,x_n \in \Sym(M)_{jp^N}$. If $n=1$, $x_1 = f(x_1)$ and we are done. Otherwise, because $M$ is $p$-typical, there has to be a partition of $\{1,\dots,n\}$ into $p$ parts $I_1,\dots,I_p$ such that for $X_i = \prod \{x_j \mid j \in I_i\}$, $|X_i| = jp^{N-1}$. Inductively, all $X_i$ are in the image of $f$, hence so is $X=f(\{X_1,\dots,X_p\})$.

We proceed to show injectivity of $\overline f$.

Let $x \in T_M$ be a monomial. We say that $y \in T_M$ occurs at depth $d$ in $x$ if either $d=0$ and $y=x$ or $x=\{x_1,\dots,x_p\}$ and $y$ occurs at depth $d-1$ in $x_i$ for some $i$.

Now suppose that $y_1$ and $y_2$ occur at a common depth $d \geq 1$ in $x=\{x_1,\dots,x_p\}$, and let $x' \in T_M$ be the element obtained by interchanging $y_1$ and $y_2$. Then I claim that $x \sim \pm x'$. To see this, we proceed by induction. If $d=1$ then the claim is true by symmetry. Suppose that $d>1$. Then $y_1$ occurs at depth $d-1$ in some $x_i$ and $y_2$ occurs at depth $d-1$ in some $x_j$. If $i=j$, we are done by induction. Otherwise, suppose without loss of generality that $i=1$ and $j=2$. Let $x_1=\{x_{11},\dots,x_{1p}\}$ and $x_2=\{x_{21},\dots,x_{2p}\}$. Without loss of generality, suppose that $y_i$ occurs at depth $d-2$ in $x_{i1}$ for $i=1,2$. Let $x^{(1)}$ be obtained from $x$ by interchanging $x_{12}$ and $x_{21}$:
\begin{align*}
x^{(1)} &=\{x^{(1)}_1,\dots,x^{(1)}_p\}\\
& = \{\{x_{11},x_{21},x_{13},\dots,x_{1p}\},\{x_{12},x_{22},\dots,x_{2p}\},x_3,\dots,x_p\}.
\end{align*}
Then $x \sim \pm x^{(1)}$ and now, $y_1$ and $y_2$ occur at depth $d-1$ in the same component $x^{(1)}_1$. By induction, $x^{(1)}_1 \sim (x^{(1)})_1' = \{x_{11}',x_{21}',x_{13},\dots,x_{1p}\}$, the element obtained from $x^{(1)}_1$ by interchanging $y_1$ and $y_2$. But then
\begin{multline*}
x \sim \pm \{\{x_{11}',x_{21}',x_{13},\dots,x_{1p}\},\{x_{12},x_{22},\dots,x_{2p}\},x_3,\dots,x_p\} \\
\sim \pm \{\{x_{11}',x_{12},\dots,x_{1p}\},\{x_{21}',x_{22},\dots,x_{2p}\},x_3,\dots,x_p\} = \pm x'.
\end{multline*}
We conclude that if $x$, $x' \in T_M$ with $f(x)=f(x')$ (i.e. they contain the same leaf elements at any given level), then $x \sim x'$. This proves the claim.
\end{proof}

\begin{corollary}\label{cor:uniqueproduct}
Let $A=A_{(j)}$ be a $p$-typical polar $k$-algebra and $x_1,\dots,x_N \in A$. Then there is a unique product $x_1\dots x_N$ in $A$ if $|x_1|+\cdots+|x_N| = jp^k$ for some $k$, and no possible way to multiply these elements in $A$ otherwise.\qed
\end{corollary}

\begin{corollary}\label{cor:evaluationofpolarpolynomials}
Let $P \in k[x_1,\dots,x_k]$ be a polynomial of degree $jp^n$ in generators $x_i$ of degree $jp^{n_i}$ for some nonnegative integers $j,n,n_1,\dots,n_k$, and let $A$ be a $p$-polar algebra. Then there is a unique and natural evaluation map
\[
\ev_P\colon A_{jp^{n_1}} \times \cdots \times A_{jp^{n_k}} \to A_{jp^n}
\]
factoring the evaluation map
\[
\hull(A)_{jp^{n_1}} \times \cdots \times \hull(A)_{jp^{n_k}} \to \hull(A)_{jp^n}; \quad (a_1,\dots,a_k) \mapsto P(a_1,\dots,a_k).
\]

\end{corollary}
\begin{proof}
We need to show that for every such $P$ and $a_1,\dots,a_k \in A$, the element $P(a_1,\dots,a_k) \in \hull(A)_{jp^n}$ lies in the subgroup $A_{jp^n}$; uniqueness is automatic because $A \to \hull(A)$ is injective by Prop.~\ref{prop:hullembedding}. Without loss of generality, we may assume that $p$ is a monic monomial.

The claim is true for $p$-typical polar $k$-algebras of type $j$ by Cor.~\ref{cor:uniqueproduct}. For an arbitrary $p$-polar $k$-algebra, note that $(a_1,\dots,a_k) \in A_{(j)}$ by assumption, thus $\ev_P(a_1,\dots,a_k) \in (A_{(j)})_{jp^n} = A_{jp^n}$.
\end{proof}




\section{Splittings of graded affine and formal groups}\label{sec:splittings}

The aim of this section is to reduce the main theorems to the case of unipotent affine groups and connected formal groups, respectively. This is achieved by studying natural splittings of
the categories of graded affine and formal groups.

In addition, we show that free affine resp. formal groups on graded $k$-algebras exist (Lemma~\ref{lemma:existenceoffree}, which is the first part of Thm.~\ref{thm:freeschemefactorization}.)

Let $k$ be a perfect field of characteristic~$p$.

\begin{prop}\label{prop:groupsplitting}
Suppose $\operatorname{char}(k)=p>2$. Then there are equivalences of categories
\begin{eqnarray*}
\AbSch{k} \simeq \AbSch{k}^{\odd} \times \AbSch{k}^u \times \AbSch{k}^m, \\
\AbSch{k}^p \simeq \AbSch{k}^{\odd} \times \AbSch{k}^u \times \AbSch{k}^{m,p},
\end{eqnarray*}
where 
\begin{itemize}
\item $\AbSch{k}^{\odd}$ denotes groups represented by Hopf algebras isomorphic to primitively generated exterior algebras on elements of odd degree,
\item $\AbSch{k}^u$ denotes evenly graded, unipotent groups, and
\item $\AbSch{k}^m$ denotes ungraded groups of multiplicative type and $\AbSch{k}^{m,p}$ the full subcategory of $p$-adic groups.
\end{itemize}
Dually, there is an equivalence of categories
\[
\FGps{k} \simeq \FGps{k}^{\odd} \times \FGps{k}^c \times \FGps{k}^e,
\]
where
\begin{itemize}
\item $\FGps{k}^{\odd}$ denotes groups represented by formal Hopf algebras isomorphic to primitively generated exterior pro-algebras on elements of odd degree,
\item $\FGps{k}^c$ denotes evenly graded, connected formal groups, and
\item $\FGps{k}^e$ denotes ungraded, \'etale formal groups.
\end{itemize}
\end{prop}

\begin{proof}
Note that the formal and the affine parts of the statement are Cartier dual to one another, thus it suffices to prove the claim in the affine case. Let $G \in \AbSch{k}$ be an affine group scheme with representing graded commutative Hopf algebra $H=\Reg{G}$.

By \cite[Prop.~A.4]{bousfield:p-adic-lambda-rings}, any bicommutative Hopf algebra (and dually, every bicommutative formal Hopf algebra) over a field of characteristic $p>2$ splits naturally into an even part and an odd part:
\[
H = H_{\ev} \otimes H_{\odd},
\]
where $H_{\ev}$ is concentrated in even degrees and $H_{\odd}$ is an exterior algebra on primitive generators in odd degrees. Thus we obtain a splitting
\[
\AbSch{k} \simeq \AbSch{k}^{\odd} \times \AbSch{k}^{\ev}
\]
Considering $G^{\ev}$, the even part of $G$, as an ungraded group scheme, it is shown in \cite[\textsection I.7]{fontaine:groupes-divisibles} that there is a natural splitting
\[
G^{\ev} \simeq G^u \times G^m
\]
into a unipotent group and a group of multiplicative type. The grading is compatible with this splitting, and the proof is complete by observing that a graded group of multiplicative type, by definition, has to be concentrated in degree $0$.

For the $p$-adic version of the statement, note that odd groups and unipotent groups are necessarily $p$-adic. 
\end{proof}

\begin{remark}
There is actually a natural splitting $G \cong G^m_0 \times G^u_0 \times G^{\ev}_{\neq 0} \times G^{\odd}$, where the three last factors are unipotent and the first two are concentrated in degree $0$.
\end{remark}

Since the free functors whose existence is claimed in Thm.~\ref{thm:freeschemefactorization} must necessarily split into components $\Fr^{\odd}$, $\Fr^u$ etc. accordingly to the splitting of $\AbSch{k}^p$ and $\FGps{k}$, the proof of that theorem is reduced to two times three cases. 

\begin{lemma}\label{lemma:multiplicativeetalefactorization}
Theorem~\ref{thm:freeschemefactorization} holds for affine schemes of multiplicative type and \'etale formal schemes.
\end{lemma}
\begin{proof}
This was done in \cite[Proof of Lemma 1.2]{bauer:p-polar} as these are ungraded.
\end{proof}

\begin{prop}\label{prop:oddfree}
Theorem~\ref{thm:freeschemefactorization} holds for odd affine and odd formal groups. More precisely, let $k$ be a field of characteristic $p>2$. Then the forgetful functors $\AbSch{k}^{\odd} \to \Alg_k^{\op}$ and $\FGps{k}^{\odd} \to (\Pro-\alg_k)^{\op}$ have left adjoints $\Fr^{\odd}$, and these factor as
\[
\begin{tikzcd}
\Alg_k^\op \arrow[r,"\Fr^{\odd}"] \ar[dr,swap,"(-)^{\odd}"] & \AbSch{k}^{\odd}\\
& (\Mod_k^{\odd})^\op \ar[u,"\hat \Fr"],\\
\end{tikzcd}
\quad \text{and} \quad
\begin{tikzcd}
\alg_k^\op \arrow[r,"\Fr^{\odd}"] \ar[dr,swap,"(-)^{\odd}"] & \FGps{k}^{\odd}\\
& (\operatorname{mod}_k^{\odd})^\op \ar[u,"\hat \Fr"],\\
\end{tikzcd}
\]
where the diagonal map assigns to $A$ the odd part of the underlying $k$-module $A$ and $\Reg{\tilde \Fr(M)} = \bigwedge(M)$.
\end{prop}

\begin{proof}
The affine case is basically a reformulation of Bousfield's result \cite[Prop.~A.4]{bousfield:p-adic-lambda-rings}, which says that the functor of primitives gives an equivalence between odd affine groups and oddly graded $k$-modules, inverse to the exterior algebra functor. To see that $\Fr^{\odd}(A)$, defined as represented by $\bigwedge(A_{\odd})$, indeed is adjoint to the forgetful functor, we calculate for any odd Hopf algebras $H = \bigwedge(M)$ with $M$ some oddly graded $k$-vector space:
\begin{multline*}
\Hom_{\Alg_k}(H,A) = \Hom_{\Alg_k}(\bigwedge(M),A)\\
 = \Hom_{\Mod_k^{\odd}}(M,A^{\odd}) = \Hom_{\Hopf{k}^{\odd}}(H,\bigwedge(A)).
\end{multline*}
The formal case follows by dualization, noting that $\bigwedge(M^*) \cong (\bigwedge(M))^*$ as formal Hopf algebras.
\end{proof}

It thus remains to show Thm.~\ref{thm:freeschemefactorization} for unipotent affine resp. connected formal groups, which in addition are even if $p>2$. The existence of the adjoints is easily established by the following lemma, but for the harder factorization result, we will need additional machinery developed in the following two sections.

\begin{lemma}\label{lemma:existenceoffree}
The forgetful functors $\AbSch{k}^u \to \Alg_k^{\op}$ and $\FGps{k}^c \to (\Pro-\alg_k)^\op$ have left adjoints.
\end{lemma}
\begin{proof}
The Hopf algebra representing $\Fr^u(A)$ is 
\[
\Cof^u(A) = \bigoplus_{n \geq 0} (A^{\otimes_k n})^{\Sigma_n},
\]
the cofree cocommutative conilpotent Hopf algebra first constructed by Takeuchi \cite{takeuchi:tangent-coalgebras}. His (ungraded) construction is compatible with the grading.

Dually, the Hopf algebra dual to the formal Hopf algebra representing $\Fr^c(A)$ is the unipotent part of the free commutative Hopf algebra on a cocommutative coalgebra \cite{takeuchi:free-hopf}.
\end{proof}

\section{Graded Witt vectors}\label{sec:witt}

In this section, we will consider commutative graded rings $A$ instead of graded-commutative rings, i.e. graded rings $A$ whose underlying ungraded ring is commutative. Similarly, a graded $p$-polar ring or algebra in this section is the analog where $\mu\colon A_j^{\otimes_k p} \to A_{jp}$ is strictly $\Sigma_p$-invariant, not just up to a sign. We will apply the results of this section to evenly graded, graded-commutative rings or graded, commutative rings over fields of characteristic $2$. There does not seem to be an adequate (for our purposes) definition of graded-commutative Witt vectors for graded-commutative rings, nor will it be necessary, in light of Prop.~\ref{prop:oddfree}

Throughout, let $p$ be a fixed prime. For a graded abelian group $M$ and an integer $i \geq 0$, we write $M(i)$ for the graded abelian group with $M(i)_n = M_{p^in}$. 

We assume the reader is familiar with the ungraded theory of $p$-typical Witt vectors, cf.~\cite{witt:witt-vectors,hazewinkel:witt,hesselholt:witt-survey}. It is tempting to define the Witt vectors $W(R)$ of a graded ring as the Witt vectors of the underlying ungraded ring and then equip it with a suitable grading, but a moment's thought shows that this doesn't quite work. The correct notion of $W_n(R)$ for graded rings is given in \cite{schoeller:Hopf,goerss:hopf-rings}. Instead of reviewing it here, we directly provide a workable definition for $p$-polar rings:

\begin{defn}
Let $A$ be a $p$-polar ring. As a graded set, the $p$-typical Witt vectors of $A$ of length $0 \leq n \leq \infty$ are defined as
\[
W_n(A) =\prod_{i=0}^n A(i), \quad \text{i.e.} \quad (W_n(A))_j = \prod_{i=0}^n A_{jp^i}.
\]
Define the ghost map $w\colon W_n(A) \to \prod_{i=0}^n A(i)$ by
\[
w(a_0,a_1,\dots) = (a_0,a_0^p+pa_1,a_0^{p^2}+pa_1^p+p^2a_2,\dots).
\]
\end{defn}

This is the same ghost map as in the definition of Witt vectors of (graded or ungraded) rings \cite{goerss:hopf-rings}. It happens to consist of $p$-polar polynomials.

\begin{lemma}\label{lemma:Wittstructurepolynomials}
Let $S_n(\mathbf x,\mathbf y) \in \Z[x_0,x_1,\dots,y_0,y_1,\dots,]$ be the polynomials in $x_0,\dots,x_n$, $y_0,\dots,y_n$ defining the addition of $p$-typical Witt vectors of rings, i.e.~such that for $\mathbf a,\mathbf b \in W(A)$ for any ring $A$,
\[
\mathbf a +_{W(A)} \mathbf b = (S_0(\mathbf a,\mathbf b),S_1(\mathbf a,\mathbf b),\dots).
\]
Then $S_n(\mathbf x,\mathbf y)$ is homogeneous of degree $jp^n$
when the variables are given the degrees $|x_i|=|y_i|=jp^i$.

If $P_n(\mathbf x,\mathbf y)$ is the corresponding polynomial for $p$-typical Witt vector multiplication, then $P_n$ is homogeneous of degree $2jp^n$ under the same hypotheses.
\end{lemma}

\begin{proof}
Since $w_n(\mathbf x,\mathbf y)$ is homogeneous of degree $jp^n$ if and only if the variables have degrees $|x_i|=|y_i|=jp^i$ (where $0 \leq i \leq n$), the claim follows from the defining properties
\[
w_n(\mathbf a +_{W(A)} \mathbf b) = w_n(\mathbf a)+w_n(\mathbf b)
\]
and
\[
w_n(\mathbf a \cdot_{W(A)} \mathbf b) = w_n(\mathbf a) w_n(\mathbf b).\qedhere
\]
\end{proof}

\begin{lemma}\label{lemma:Wittofppolar}
Let $A$ be a graded $p$-polar ring and $1 \leq n \leq \infty$.
\begin{enumerate}
\item There is a unique functorial graded $p$-polar ring structure on $W_n(A)$ making the ghost map $w$ into a $p$-polar ring map. If $A=\pol_p(B)$ for some graded ring $B$ and $W_n(B)$ denotes the graded Witt vectors of \cite{goerss:hopf-rings}, then there is a natural isomorphism $\pol_p(W_n(B)) \cong W_n(\pol_p(B))$.
\label{lemma:Wittofppolar:ppolarstructure}
\item \label{lemma:Wittofppolar:FV} There are unique natural additive maps $F\colon W_{n+1}(A) \to W_n(A)(1)$ (Frobenius) and $V\colon W_{n}(A)(1) \to W_{n+1}(A)$ (Verschiebung) such that the following diagram commutes ($x_{-1}=0$ by convention):
\[
\begin{tikzcd}[column sep={10em,between origins}]
W_{n}(A)(1) \arrow[d,"w"] \arrow[r,"V"] & W_{n+1}(A) \arrow[d,"w"] \arrow[r,"F"] & W_{n}(A)(1) \arrow[d,"w"]\\
\prod_{i=0}^{n-1} A(i+1) \arrow[r,"(x_i)_i \mapsto (px_{i-1})"] &
\prod_{i=0}^{n} A(i) \arrow[r,"(x_i)_i \mapsto (x_{i+1})_i"] & \prod_{i=0}^{n-1} A(i+1)
\end{tikzcd}
\]
Explicitly, $V(a_0,\dots,a_{n-1}) = (0,a_0,\dots,a_{n-1})$ and $F$ is uniquely determined by $F(\underline a) = \underline{a^p}$ and $FV=p$, where $\underline a = (a,0,\dots,0)$ denotes the Teichm\"uller representative of $a \in A$ in $W_n(A)$.

If $A=\pol_p(B)$ then $F$ and $V$ coincide with the operators of the same name defined on $W_n(B)$ \cite{goerss:hopf-rings}.

\item \label{lemma:Wittofppolar:algebra} If $A$ is a $p$-polar algebra over a perfect field $k$ of characteristic $p$ then $W_n(A)$ is a $p$-polar $W(k)$-algebra. For a graded $W(k)$-module $M$, give $M(i)$ the $W(k)$-module structure defined by
\[
\lambda . m = \frob^n(\lambda)m,
\]
where $\frob$ is the Frobenius on $W(k)$. Then $F$ and $V$ are $W(k)$-linear.
\end{enumerate}
\end{lemma}

\begin{proof}
This was proved in the ungraded case in \cite[Lemma~3.3]{bauer:p-polar}, so assume $j>0$.

For $\mathbf a,\;\mathbf b \in W_n(A)_{jp^i}$, we have
\[
\mathbf a +_{W_n(A)} \mathbf b = (S_0(\mathbf a,\mathbf b),\dots,S_n(\mathbf a,\mathbf b)).
\]
Note that this is well-defined by Lemma~\ref{lemma:Wittstructurepolynomials} and Cor.~\ref{cor:evaluationofpolarpolynomials}. 
Similarily, if $\mathbf a^{(1)},\dots,\mathbf a^{(p)} \in W_n(A)_{jp^i}$, since the polynomial $P_n(\mathbf x^{(1)},\dots,\mathbf x^{(p)})$ for $p$-fold multiplication has degree $np^{i+1}$, evaluation provides a well-defined $p$-polar multiplication on $W_n(A)$. 

Since $w$ is injective in the case where $A$ is torsion-free, this $p$-polar ring structure is unique and is compatible with the ring structure on $W_n(B)$ when $B$ is a graded ring.

This proves \eqref{lemma:Wittofppolar:ppolarstructure}. Parts \eqref{lemma:Wittofppolar:FV} and \eqref{lemma:Wittofppolar:algebra} follow easily.
\end{proof}

If $A$ is a strong $p$-polar graded algebra and $A^{\ungr}$ denotes $A$ as an ungraded $p$-polar algebra then $W(A^{\ungr})$ and $W(A)^{\ungr}$ are in general distinct:
\begin{example}
	$W(k[u]) \cong W(k)[u]$ if $|u| =d> 0$. This is false if $d=0$: the (ungraded) $p$-typical Witt vectors of $W(\F_p[x])$ are more complicated (cf. \cite[Exercise~10]{borger:witt-vector-lectures}). The $p$-polar algebra $W(A^{\ungr})$ can be given a natural $\Z[\frac1p]$-grading, and $W(A)$ can be identified with the sub-$p$-polar algebra concentrated in integer degrees.
\end{example}

The functor $M \mapsto M(1)$ is not an equivalence. But if $k$ is a perfect field (and hence $\frob$ is invertible), it has a right inverse $(-1)$ given by
\[
M(-1)_n = \begin{cases} 0;& p \nmid n\\
M_{\frac n p}; & p \mid n\end{cases}
\]
with the $W(k)$-linear structure given by $\alpha.m = \frob^{-1}(\alpha)m$. Confusingly, $(-1)$ being a right inverse means that $M(-1)(1)\cong M$.

\begin{defn}
For a graded $p$-polar ring $A$, the group of unipotent co-Witt vectors is defined by
\[
CW^u(A) = \colim(W_0(A) \xrightarrow{V} W_1(A(-1)) \xrightarrow{V} W_2(A(-2)) \xrightarrow{V} \cdots),
\]
where $V\colon W_n(A) = W_n(A(-1))(1) \to W_{n+1}(A(-1))$ is induced by the Verschiebung $V\colon W_n(A)(1) \to W_{n+1}(A)$.
\end{defn}
If $A$ is a graded $p$-polar $k$-algebra for a perfect field $k$ of characteristic $p$ then $CW^u(A)$ is naturally a $W(k)$-module.

\subsection{Representability of Witt and co-Witt vectors}

Since for graded $k$-algebras $A$ and $n \leq \infty$, $W_n(A)_j \cong \prod_{i=0}^n A_{jp^i}$ as sets, this set-valued functor is represented in graded $k$-algebras by
\[
(\Lambda^{(n)}_p)_j=k[\theta_{j,0},\dots,\theta_{j,n}]; \quad (\Lambda^{(\infty)}_p)_j = (\Lambda_p)_j = k[\theta_{j,0},\theta_{j,1},\dots]
\]
where $|\theta_{j,i}| = jp^i$, and $W_n(A)$, as a graded object, is represented by the bigraded $k$-algebra $\Lambda^{(n)}_p = (\Lambda^{(n)}_p)_*$. Each $(\Lambda^{(n)}_p)_j$ obtains a Hopf algebra structure by the natural Witt vector addition on $W_n(A)_j$, and $\Lambda^{(n)}_p$ becomes a Hopf ring (cf. \cite{wilson:hopf-rings}) with a comultiplication
\[
\Lambda_p(A)_j \to \bigoplus_{j_1+j_2=j} \Lambda_p(A)_{j_1} \otimes \Lambda_p(A)_{j_2}.
\]
In other words, $\Lambda_p$ represents a graded ring object, even a plethory, in affine schemes. This is a graded version of the $p$-typical symmetric functions of \cite[II.13]{borger-wieland:plethystic}.

The co-Witt vectors are not representable, but their restriction to $\alg_k$, i.e. finite-dimensional $k$-algebras, is ind-representable. That is, $CW^u_k$ is a formal group. In the ungraded case, this is described in \cite[\textsection II.3--4]{fontaine:groupes-divisibles}. In our graded case, $(W_n)_j$ is represented by $k\pow{\theta_{j,0},\dots,\theta_{j,n}}$ as before; we choose a power series notation as we would in the ungraded case, although for graded algebras of finite type, there is arguably no difference.

Thus $CW^u(A)_j$ is represented by the profinite ring 
\[
(\Reg{CW^u_k})_j = 
\lim_{n \geq 0} k\pow{x_{j,0},x_{j,-1},\dots,x_{j,-n}},
\]
with $|x_{j,i}| = jp^i$, which in the case $jp^i \not\in\Z$ is to be understood as $x_{j,i}=0$. Note that for $j = a p^l \neq 0$ with $p \nmid a$, the above formula simplifies to
\[
(\Reg{CW^u_k})_j = k\pow{x_{j,0},\dots,x_{j,-l}}.
\]

By naturality of the co-Witt vector addition, $CW^u_k$ thus becomes a (graded) formal group.

\begin{example}
Let $k=k_0$ be perfect of characteristic $p$ and $A = k \langle x^{p^i} \mid i \geq 0\rangle$ the free $p$-polar algebra on a single generator $x$ in degree $2$. Then
\[
W_n(A)_{2p^i} = \{(a_0 x^{p^i},a_1 x^{p^{i+1}},\dots,a_n x^{p^{i+n}}) \mid a_i \in W(k)\} \cong W_n(k)
\]
and $W_n(A)_j=0$ if $j$ is not twice a power of $p$. The Verschiebung is given by
\[
V\colon W_n(A)_{2p^i} \to W_{n+1}(A(-1))_{2p^i}, \quad (a_0,\dots,a_n) \mapsto (0,a_0,\dots,a_n).
\]
We have that
\[
CW^u(A)_{2p^i} = \{ (\dots,a_{-1},a_0) \mid a_j \in A_{2p^{i+j}}\},
\]
which is understood to mean $a_j=0$ if $2p^{i+j} \not\in\Z$. Thus $CW^u(A)_{2p^i} = W_i(k)$, and the Frobenius and Verschiebung maps are given by
\[
V\colon CW^u(A)_{2p^i} \to CW^u(A)_{2p^{i-1}}, \text{ the restriction map $W_i(k) \to W_{i-1}(k)$}
\]
and
\[
F\colon CW^u(A)_{2p^i} \to CW^u(A)_{2p^{i+1}}, \text{ the multiplication-by-$p$ map $W_i(k) \to W_{i+1}(k)$.}
\]
\end{example}

\section{Graded Dieudonné theory}\label{sec:dieudonne}

Throughout, let $k$ be a perfect field of characteristic $p$. Recall that for a graded $W(k)$-module $M$, the module $M(1)$ has
$M(1)_k = M_{pk}$ with $W(k)$-action given by $\alpha.m = \frob(\alpha)m$, where $\frob$ is the Frobenius endomorphism of $W(k)$.

\begin{defn}
A \emph{graded Dieudonné module} over $k$ is a graded $W(k)$-module $M$ together with maps of $W(k)$-modules
\[
F\colon M \to M(1) \quad \text{and} \quad V\colon M(1) \to M
\]
satisfying $FV=p$ and $VF=p$.
We denote the category of Dieudonné modules (with the obvious definition of morphism) by $\DMod{k}$.

We call a Dieudonné module $M$ \emph{unipotent} if for any $x\in M$, $V^n(x)=0$ for $n \gg 0$. We denote the full subcategory of unipotent Dieudonné modules by $\DModV{k}$. Note that for degree reasons, the unipotence condition can only possibly fail for $x$ of degree $0$.

Moreover, a Dieudonné module $M$ is called \emph{connected} if $M$ is profinite as a $W(k)$-module and has a fundamental system of neighborhoods consisting of (finite length) $W(k)$-modules $N$ such that $F^n(N)=0$ for $n \gg 0$. We denote the subcategory of connected Dieudonné modules and continuous homomorphisms by $\DDModF{k}$
\end{defn}

\begin{proof}[Proof of Lemma~\ref{lemma:wittofppolar}]
Lemma~\ref{lemma:Wittofppolar}\eqref{lemma:Wittofppolar:ppolarstructure} shows in particular that $W_n\colon \Alg_k \to \Ab$ factors through $\Pol_p(k)$, and part \eqref{lemma:Wittofppolar:FV} of the same lemma shows that this even holds as Dieudonné modules with the canonical operators $F$ and $V$ on $W_n$. Thus $W_n$ factors through $\Pol_p(k)$ for all $n \leq \infty$.

Moreover, we have a commutative diagram of abelian groups, natural in $B$,
\[
\begin{tikzcd}
 W_n(\pol_p(B)) \ar[r,"V"] \ar[d] & W_{n+1}(\pol_p(B)) \ar[d]\\
 W_n(B) \ar[r,"\pol_p(V)"] & W_{n+1}(B)
\end{tikzcd}
\]
Thus $CW^u(\pol(B)) \cong \colim_V CW_n(\pol(B)) \cong \colim_V CW_n(B) \cong CW^u(B)$ as Dieudonné modules and $CW^u$ also factors through $\Pol_p(k)$.

For $CW^c$, the factorization in degree $0$ was proved in \cite[Theorem~1.4]{bauer:p-polar}, and the claim in positive degrees follows from the result about $CW^u$ above and the fact that $CW^{c,u}(R)<CW^u(R)$ consists of the co-Witt vectors whose entries lie in the nilradical $\nil(R)$, which is a $p$-polar notion: $r^n=0$ for some $n\geq 1$ iff $r^{p^m}=0$ for some $m\geq0$.
\end{proof}

\begin{thm} \label{thm:generaldieudonne}
There are exact natural equivalences
\[
\DieuFor\colon \FGps{k}^c \to \DDModF{k} \quad \text{and} \quad D\colon \AbSch{k}^u \to \DModV{k}.
\]
The functor $\DieuFor$ is represented by the formal group $CW^c_k$, while the functor $D(G)$ is given by $D(G)=\colim_n \Hom_{\AbSch{k}}(G,W_n)$.
\end{thm}
\begin{proof}
Since every affine group splits naturally into its degree-$0$ part and a part represented by a Hopf algebra $H$ that is connected in the sense that $H_0=k$, the result follows from the conjunction of \cite[\textsection III Théorème 1]{fontaine:groupes-divisibles} and \cite{schoeller:Hopf}, cf. also \cite[Theorem~4.2]{bauer:p-polar} and \cite{goerss:hopf-rings,ravenel:dieudonne}.

The equivalence $\DieuFor$ is given by the composition
\[
\FGps{k}^c  \xrightarrow{(-)^*} \AbSch{k}^{u}\xrightarrow{D} \DModV{k}\xrightarrow{I} (\DDModF{k}),
\]
where $(-)^*$ denotes Cartier duality and $I$ denotes Matlis (Poincaré) duality
\[
M \mapsto \Hom^c_{W(k)}(M,CW(k)),
\]
the group of continuous homomorphisms into $CW(k)$. Since both $(-)^*$ and $I$ are anti-equivalences, the affine case follows from the formal case. The fact that $\DieuFor$ is isomorphic to the claimed form is proved in \cite[Ch.~III,6]{demazure:pdivisiblegroups}.
\end{proof}

\begin{proof}[Proof of Thm.~\ref{thm:dieudonneoffree}]
Consider first the formal case. For a finite-dimensional $k$-algebra $A$, we have that
\[
\DieuFor(\Fr^c(A)) = \Hom_{\FGps{k}}(\Fr^c(A),CW^c_k) = \Hom_{\FSch{k}}(\Spec A, CW^c_k) = CW^c(A),
\]
as claimed.

In the affine case, for a $k$-algebra $A$, we have
\[
D(\Fr^u(A)) = \colim_n \Hom_{\AbSch{k}}(\Fr(A),W_n) = \colim W_n(A) = CW^u(A). \qedhere
\]
\end{proof}

\begin{proof}[Proof of Thm.~\ref{thm:freeschemefactorization}]
In the formal case, the existence of the adjoint $\Fr = \Fr^{\odd} \times \Fr^c \times \Fr^e$ follows from Lemma~\ref{lemma:existenceoffree}. The factorization $\tilde \Fr(A)$ for $A \in \pol_p(k)$ is given by:
\begin{itemize}
	\item $\tilde \Fr^{\odd}(A)$ is defined as in Prop.~\ref{prop:oddfree};
	\item $\tilde \Fr^e(A) = \Fr^e(A_0)$ is defined as in \cite{bauer:p-polar};
	\item $\tilde \Fr^c(A) = (\DieuFor)^{-1}(CW^c(A))$, using Thm.~\ref{thm:dieudonneoffree}.
\end{itemize}

The affine case is completely analogous.
\end{proof}

\section{Properties and applications}\label{sec:applications}

The factorization of the free formal group functor (Thm.~\ref{thm:freeschemefactorization}) induces a factorization
\[
\begin{tikzcd}
(\Pro-\alg_k)^{\op} \arrow[r,"\Fr"] \ar[dr,swap,"\pol"] & \FGps{k}\\
& (\Pro-\pol_p(k))^{\op}, \ar[u,"\tilde \Fr"]\\
\end{tikzcd}
\]
where the functors denoted by $\Fr$ and $\tilde \Fr$ are the unique extension of the functors from Thm.~\ref{thm:freeschemefactorization} that commute with directed colimits. In this section, we will concentrate on the free unipotent, resp. connected, construction only.

\begin{lemma}\label{lemma:formalfreeadjoint}
The functor $\hat \Fr{}^c\colon (\Pro-\pol_p(k))^{\op} \to \FGps{k}$ commutes with all colimits and has a right adjoint $V$.
\end{lemma}
\begin{proof}
For the odd part, the functor $\hat\Fr^{\odd}$ factors as
\[
\hat\Fr{}^{\odd}\colon (\Pro-\pol_p(k))^{\op} \xrightarrow{U} (\Pro-\operatorname{mod}_k^{\odd})^{\op} \xrightarrow{\bigwedge} \AbSch{k}^p
\]
(cf. Prop.~\ref{prop:oddfree}), where $U$ is the forgetful functor. Then an adjoint is given by the composition of the adjoint of the exterior Hopf algebra functor $\bigwedge$ (which is the functor of primitives) and an adjoint of $U$. The latter is the objectwise free $p$-polar algebra functor, which works because the free $p$-polar algebra on an odd finite-dimensional $k$-module is again finite dimensional (it is a sub-$p$-polar algebra of the exterior algebra).

So we can restrict our attention to even formal groups. For the free connected formal group $\Fr^c$, it suffices to show that its composition with the Dieudonné equivalence $\DieuFor$ commutes with the stated colimits, and by Theorem~\ref{thm:dieudonneoffree}, it is therefore enough to show that $CW_k^c\colon \Pro-\pol_p(k) \to \DDModF{k}$ commutes with all limits. But $CW_k^c$ is a formal group, i.e. representable. 

The existence of an adjoint follows from Freyd's special adjoint functor theorem once we show that $\Pro-\pol_p(k)$ is complete, well-powered, and possesses a cogenerating set. Any pro-category of a finitely complete category, such as $\pol_p(k)$, is complete, and any pro-category has the constant objects as a cogenerating class. Since $\pol_p(k)$ has a small skeleton, the condition on a cogenerating set is satisfied.
To see that $\Pro-\pol_p(k)$ is well-powered, observe that a subobject $S < A$ for $A \in \Pro-\pol_p(k)$ is in particular a sub-pro-vector space. By \cite[Prop. 4.6]{artin-mazur:etale}, a monomorphism in $\Pro-\Mod_k$ can be represented by a levelwise monomorphism. Thus if $A\colon I \to \Mod_k$ represents a pro-finite $k$-module with $\#\operatorname{Sub}(A(i)) = \alpha_i$ for some (finite) cardinals $\alpha_i$ then $\#\operatorname{Sub}(A) \leq \prod_{i \in I} \alpha_i$; in particular, it is a set.
\end{proof}

\begin{lemma}\label{lemma:affinefreeadjoint}
The functor $\hat \Fr{}^u\colon \Pol_p(k)^{\op} \to \AbSch{k}^u$ commutes with all colimits and filtered limits, and has a right adjoint $V$.
\end{lemma}
\begin{proof}
As in Lemma~\ref{lemma:formalfreeadjoint}, the right adjoint on odd affine groups is given by the functor of primitives follow by the free $p$-polar algebra functor, so we will restrict our attention to even $p$-adic affine groups. By Theorem~\ref{thm:dieudonneoffree}, it suffices to show that the functors $CW^u\colon \Pol_p(k) \to \DModV{k}$ commutes with all limits, which looks wrong until one realizes that limits in $\DModV{k}$ are not the same as limits in $\DMod{k}$.

The functor $CW^u(A)$ commutes with finite limits (it is ind-representable), so it suffices to show it commutes with infinite products. Indeed, the natural map
\[
CW^u(\prod_i A_i) \to \prod_i CW^u(A_i)
\]
is an isomorphism; an element in the right hand side is a set of elements $(x_i \in CW^u(A_i))$ such that there is an $n \gg 0$ such that $V^n(x_i)=0$ for all $i$. 



The commutation with filtered colimits is straightforward: $CW^u$ commutes with them because it is a colimits of functors represented by small objects (polarizations of finitely presented algebras).

For the existence of an adjoint, we again apply the special adjoint functor theorem in the form of \cite[Thm.~1.66]{adamek-rosicky:presentable-accessible}. The category $\Pol_p(k)$ is locally presentable and the functor $D \circ \hat Fr^u$, as just shown, is accessible (commutes with $\omega$-filtered colimits) and commutes with all limits.
\end{proof}

Note that this implies, by taking adjoint functors in Thm.~\ref{thm:freeschemefactorization}, that the algebra underlying a unipotent Hopf algebra $H$ is always of the form $\hull(V(H))$, i.e. free over a $p$-polar $k$-algebra, and the pro-finite algebra underlying a formal connected Hopf algebra $H$ is always of the form $\hull(V(H))$, i.e. free over a profinite $p$-polar $k$-algebra. Of course, this is also a direct corollary of Borel's work on the structure of algebras underlying Hopf algebras \cite{borel:homologie-des-groupes-de-Lie,milnor-moore:hopf} .

Let us call a graded module $M$ \emph{$p$-typical of degree $j$} if $M_n=0$ unless $n=jp^i$ for some $i$ and denote their full subcategory by $\Mod_k^{(j)}$. Recall from \cite{ravenel:dieudonne} that an abelian Hopf algebra over $k$ is called of \emph{type $j$} if its modules of primitives and indecomposables are $p$-typical of degree $j$, or equivalently, if its Dieudonné module is $p$-typical of degree $j$. Denote the category of Hopf algebras of type $j$ by $\Hopf{k}^{(j)}$. Similarly, let $\Pol_p^{(j)}(k)$ denote the category of $p$-typical $p$-polar $k$-algebras of degree~$j$.

\begin{lemma}\label{lemma:Vofptypical}
If $A \in \Pol_p^{(j)}(k)$ then $\Cof^u(A) \in \Hopf{k}^{(j)}$. Conversely, if $H \in \Hopf{k}^{(j)}$ then $V(H) \in \Pol_p^{(j)}(k)$, and if $j \neq 0$, $V(H) = H_{(j)}$.
\end{lemma}
\begin{proof}
If $A \in \Pol_p^{(j)}(k)$ then $CW^u(A)$ is $p$-typical by definition, hence $\Cof^u(A) = D^{-1}CW^u(A)$ is of type $j$.

Let $H \in \Hopf{k}^{(j)}$ and $A \in \Pol_p^{(j')}(k)$ a $p$-typical $p$-polar algebra of a possibly different degree $j'$. Then
\[
\Hom_{\Pol_p(k)}(VH,A) \cong \Hom_{\Hopf{k}}(V,\Cof(A)) = 0 \text{ unless $j'=j$}
\]
since $\Cof(A)$ has type $j'$ and there are no nontrivial maps between Hopf algebras of different types. Hence $VH \in \Pol_p^{(j)}(k)$.

If $j \neq 0$ then for $P \in \Pol^{(j)}_p(k)$, $P = \hull(P)_{(j)}$. In particular, for $H \in \Hopf{k}^{(j)}$,
\[
VH = \hull(VH)_{(j)} = H_{(j)}. \qedhere
\]
\end{proof}

In the ungraded case, it is harder to describe $V(H)$ explicitly. However, there is a cyclic grading by the group $\Z/(p-1)\Z$ implicit in the construction of the functor $\hull$. Let $\Alg_k^c$, $\Hopf{k}^c$, $\Pol_p^c$ be the $\Z/(p-1)$-graded versions of $\Alg_k$, $\Hopf{k}$, and $\Pol_p$. Let $\Pol_p^{c,1}$ be the direct factor full subcategory of cyclically graded $p$-polar algebras that are concentrated in degree $1$, and let $\Hopf{k}^{c,1}$ be the full subcategory of Hopf algebras all of whose primitives and decomposables are in degree $1$. This is, of course, canonically equivalent to the category $\Pol_p^{(0)}(k)$ of ungraded $p$-polar $k$-algebras.

There is a natural factorization $\hull^c$ of the functor $\hull$:
\[
\begin{tikzcd}
\Pol_p^{c,1} \ar[r,"\hull^c"] \ar[d,swap,"\cong"] & \Alg_k^{c} \ar[d,"U"]\\
\Pol_p^{(0)}(k) \ar[r,"\hull"] & \Alg_k^{(0)},
\end{tikzcd}
\]
where $U$ is the functor forgetting the grading.

The functor $\hull^c$ is left adjoint to the functor
\[
(-)_1\colon \Alg_k^{c} \to \Pol^{c,1}; \quad A \mapsto A_1.
\]
Note that $U$ does not have a left adjoint.

We again have a natural isomorphism $P \cong (\hull^{c}(P))_1$ for $P \in \Pol_p^{c,1}$. The cofree functor $\Cof^u\colon \Pol_p(k) \to \Hopf{k}$ respects the cyclic grading:
\[
\Cof^{u,c}\colon \Pol_p^{c} \cong \Pol_p^{c,1} \to \Hopf{k}^{c},
\]
and the cyclically graded analog of Lemma~\ref{lemma:Vofptypical} is:
\begin{lemma}\label{lemma:Vcyclicallygraded}
If $A \in \Pol^{(0)}_p(k)$ then $\Cof^{u,c}(A) \in \Hopf{k}^{c,1}$. Conversely, if $H \in \Hopf{j}^{c,1}$ then $V(H) \in \Pol_{p}^{c,1} \cong \Pol_{p}^{(0)}(k)$, and $V(H) \cong H_1$. \qed
\end{lemma}

\begin{lemma}\label{lemma:cofreeiscofreeascoalg}
Let $A$ be a $p$-polar $k$-algebra and $H=\Cof^u(A)$ the unipotent Hopf algebra representing the $p$-adic affine group $\hat \Fr{}^u(A)$. Then $H$ is isomorphic, as a pointed coalgebra, to the symmetric tensor coalgebra on the $k$-vector space $A$.
\end{lemma}
\begin{proof}
The symmetric tensor coalgebra on a vector space $V$ is given by
\[
S(V) = \bigoplus_{i \geq 0} S^i(V) \quad \text{with} \quad S^i(V)= \bigl(V^{\otimes i}\bigr)^{\Sigma_i},
\]
and it is a pointed coalgebra by the inclusion $k \cong S^0(V) \subset S(V)$. A pointed coalgebra $C$ is conilpotent if for each $x \in C$ there exists an $N> 0$ such that $\psi^N(x) \in C^{\otimes (N+1)}$ maps to $0$ in $\bar C^{\otimes (N+1)}$, where $\bar C$ is the cokernel of the pointing. The coalgebra $S(V)$ is conilpotent and, indeed, the right adjoint to the forgetful functor $U$ from the category $\Coalg_k^u$ of pointed, conilpotent, cocommutative coalgebras to $k$-vector spaces, mapping a coalgebra $C$ to $\bar C$.

Note that a Hopf algebra is unipotent if and only if its underlying pointed coalgebra is conilpotent. The claim is that the diagram
\begin{equation}\label{eq:rightadjointsquare}
\begin{tikzcd}
\Pol_p(k) \arrow[r,"\Cof^u"] \arrow[d,"U_1"] & \Hopf{k}^u \arrow[d,"U_2"]\\
\Mod_k \arrow[r,"S"] & \Coalg_{k}^u
\end{tikzcd}
\end{equation}
$2$-commutes. Since all the functors in this diagram are right adjoints, they commute with limits, and thus by Rem.~\ref{remark:ptypicalsplitting} it suffices to show the claim in the cases where $A$ is $p$-typical of degree $j \neq 0$ and where $A=A_0$ is ungraded.

By Lemma~\ref{lemma:Vofptypical}, Diagram \eqref{eq:rightadjointsquare} restricts to a diagram
\[
\begin{tikzcd}{}
\Pol_p^{(j)}(k) \arrow[r,"\Cof^u"] \arrow[d,"U_1"] & \Hopf{k}^{(j)} \arrow[d,"U_2"]\\
\Mod_k^{(j)} \arrow[r,"S"] & \Coalg_{k}^{(j)}
\end{tikzcd}
\]
where $\Coalg_k^{(j)}$ denotes conilpotent coalgebras whose module of primitives is $p$-typical of type $j$.

By taking left adjoint functors, again using Lemma~\ref{lemma:Vofptypical}, the $2$-commutativity of this diagram is equivalent with the $2$-commutativity of
\[
\begin{tikzcd}
\Pol_p^{(j)}(k) & \Hopf{k}^{(j)} \arrow[l,swap,"V"] \\
\Mod_k^{(p)}  \ar[u,"\Fr"] & \Coalg_{k}^{(j)}. \arrow[l,"U"] \arrow[u,"\Fr"]
\end{tikzcd}
\]

In the case $j\neq 0$, we know by Lemma~\ref{lemma:Vofptypical} that $VH=H_{(j)}$ and the claim follows from the fact that $\Sym(\bar C)_{(j)}$ is the free $p$-polar algebra on $U(C)=\bar C$ (cf. the Claim in Lemma~\ref{lemma:nonperiodicchar}).

It remains to consider the ungraded case $j=0$. Observe that using Lemma~\ref{lemma:Vcyclicallygraded}, diagram \eqref{eq:rightadjointsquare} can be refined to a diagram of $\Z/(p-1)$-graded objects:
\[
\begin{tikzcd}
\Pol_p^{(0)}(k) \arrow[r,"\Cof^{c,u}"] \arrow[d,"U_1"] & \Hopf{k}^{c,1} \arrow[d,"U_2"]\\
\Mod_k^{(0)} \arrow[r,"S"] & \Coalg_{k}^{c,1}.
\end{tikzcd}
\]

By taking adjoints as before, we obtain
\[
\begin{tikzcd}
\Pol_p(k)^{(0)} & \Hopf{k}^{c,1} \arrow[l,swap,"V^c"] \\
\Mod_k ^{(0)} \ar[u,"\Fr"] & \Coalg_{k}^{c,1}. \arrow[l,"U"] \arrow[u,"\Fr"]
\end{tikzcd}
\]
where $U(C)=C_1$. For $P$ a $p$-polar algebra concentrated in degree $1$, we have that $\hull^c(P)_{1} = P$, and hence $V^c(\Fr(C)) = \Fr(C)_{1} = \Sym(\bar C)_{1}$ is again the free $p$-polar algebra on $\bar C$.
\end{proof}

\begin{corollary}\label{cor:primitiveofcofree}
Let $H$ be a unipotent Hopf algebra which is unipotent cofree on a $p$-polar $k$-algebra $A$. Then $A$ is isomorphic to the vector space of primitive elements $PH$.
\end{corollary}
\begin{proof}
By the preceding Lemma, $U_2(H) \cong S(U_1(A))$. Applying the functor $P$ and noting that $P(S(M)) \cong M$, we find that $PH \cong U_1(A)$.
\end{proof}

\begin{remark}
It is not true that $V(H)=P(H)$ in general, or that $P(H)$ is a $p$-polar algebra. Also, if $H$ is a unipotent Hopf algebra whose underlying pointed unipotent coalgebra is cofree, $H$ is not necessarily cofree over a $p$-polar $k$-algebra. For example, consider the graded Hopf algebra $H$ dual to $H^*=k[x,y]$ with $|x|=j>0$, $|y|=p^2j$, $x$ primitive and $\psi(y) = y \otimes 1 + 1 \otimes y + \sum_{i=1}^{p-1} \frac1{i!(p-i)!} x^{pi} \otimes x^{p(p-i)}$. Then the primitives $PH$ are dual to the indecomposables $Q(H^*) = \langle x,y\rangle$, i.e. $PH = \langle a,b\rangle$ with $|a|=j$, $|b|=p^2j$. Suppose $H$ was cofree, so that $PH$ is a $p$-polar $k$-algebra by the corollary above. For degree reasons, $PH$ cannot carry any but the trivial $p$-polar algebra structure, $PH \cong \langle a \rangle \times \langle b \rangle$. As a right adjoint, $\Cof^u$ commutes with products and hence $\Cof^u(PH) = \Cof(\langle a \rangle) \otimes \Cof(\langle b \rangle) = (k[x] \otimes k[y])^* = H'$. But $H \not\cong H'$ as Hopf algebras since $P(H^*) = \langle x,x^p,x^{p^2},\dots \rangle$, while $P((H')^*) = \langle x,x^p,\dots,y,y^p,\dots\rangle$, a contradiction.
\end{remark}
\begin{proof}[Proof of Thm.~\ref{thm:lambdapcofree}]
Let $\Lambda_p = k[\theta_{j,0},\theta_{j,1},\dots]$ be the Hopf algebra representing the functor of $p$-typical Witt vectors. Denote by $\eta\colon \Lambda_p \to \Cof^u(V(\Lambda_p))$ the unit of the adjunction. Applying the functor of primitives, since $\Cof^u(V(\Lambda_p))$ is cofree as a coalgebra by Lemma~\ref{lemma:cofreeiscofreeascoalg}, we obtain a map
\[
P\eta\colon P(\Lambda_p) \to V(\Lambda_p).
\]
Now $V(\Lambda_p) = \pol_{(j)}(\Lambda_p) = k\langle \theta_{j,i}^{p^k} \mid i,k \geq 0\rangle$ and $P(\Lambda_p) = k\langle \theta_{j,0}^{p^k} \mid k \geq 0\rangle$, and $P\eta$ is the inclusion map. We see that $P(\Lambda_p)$ is in fact a direct factor of $V(\Lambda_p)$ as a $p$-polar algebra, with a retraction $p\colon V(\Lambda_p) \to P(\Lambda_p)$ given by
\[
p(\theta_{j,i}) \mapsto \begin{cases} \theta_{j,0}; & i=0\\
0; & \text{otherwise} \end{cases}
\]
Taking adjoints, we obtain a map of Hopf algebras $q\colon \Lambda_p \to \Cof^u(P\Lambda_p)$. A map of unipotent Hopf algebras is injective iff it is injective on primitives, and $Pq\colon P\Lambda_p \to P(\Cof^i(P\Lambda_p))\cong P\Lambda_p$ is the identity, so $q$ is injective. By dimension considerations, it must also be surjective.
\end{proof}

\subsection{\texorpdfstring{$F$-modules and $p$-polar rings}{F-modules and p-polar rings}}

\begin{defn}
An \emph{$F$-module} is a positively graded profinite $k$-vector space $N$ with a $k$-linear map $F\colon N \to N(1)$. Denote their category by $\Mod_F$. 
\end{defn}

Note that because of the grading, we could equivalently have defined an $F$-module to be positively graded and of finite type.

The category $\Mod_F$ is anti-equivalent (by taking continuous duals) to the category $\Mod_V$ of $k$-vector spaces $M$ with a $k$-linear map $V\colon M(1) \to M$.

Kuhn showed in \cite[Thm.~2.11]{kuhn:quasi-shuffle} that every $V$-module of finite type over a perfect field $k$ is a sum of modules $M(j,m)$ as follows:
\[
M(2j,m) = \langle x,V^{-1}x,\dots,V^{-m} x\rangle \quad \text{with } |V^{-i}x|=2p^ij, V(V^ix)=V^{i+1}x, Vx=0,
\]
\[
M(2j,\infty) = \langle x,V^{-1}x,\dots\rangle \quad \text{with } |x|=2j,
\]
and
\[
M(2j+1,0) = \langle y \rangle \quad \text{with } |x|=2j+1, \; V=0.
\]

\begin{remark}
Finite type cannot be relaxed, but in the case of countable generation, Ulm's theorem allows a somewhat more complicated classification of $V$-modules \cite[Thm. 2 and the remarks following the proof]{webb:graded-modules}. 
\end{remark}

\begin{corollary}\label{cor:structureofFmodules}
If $k$ is perfect, any $F$-module of finite type is isomorphic to a product of $N(j,m)=M(j,m)^*$. \qed
\end{corollary}

There is a forgetful functor $U_F\colon \Pro-\pol_p(k) \to \Mod_F$ from profinite $p$-polar $k$-algebras to $F$-modules, given by $U_F(A)=(A,x \mapsto x^p)$.

\begin{lemma}\label{lemma:Fmodulelift}
For every $M \in \Mod_F$ of finite type, there exists an $A \in \Pro-\pol_p(k)$ such that $U_F(A) \cong M$. 
\end{lemma}
\begin{proof}
Modules of the form $N(j,m)$ carry a unique $p$-polar algebra structure $A(j,m)$. For an arbitrary $F$-module $M$, which by Cor.~\ref{cor:structureofFmodules} is isomorphic to a product $\prod_{i \in I} N(j_i,m_i)$, define $A = \prod_{i \in I} A(j_i,m_i)$. Then $U_F(A) \cong M$.
\end{proof}

Note that this construction is not functorial because it depends on the chosen decomposition of $M$. It is not true that any two $p$-polar algebras with isomorphic underlying $F$-module are isomorphic.

We obtain the following slight variation of Kuhn's theorem:

\begin{corollary}
Given any $F$-module $M$ of finite type, there is a formal connected Hopf algebra $H$ which is cofree over a profinite $p$-polar algebra and such that $PH\cong M$, unique up to isomorphism of formal Hopf algebras cofree over profinite $p$-polar algebras.
\end{corollary}
\begin{proof}
Given $M$, choose a $p$-polar algebra structure on $M$ as in Lemma~\ref{lemma:Fmodulelift} and define $H=\Cof^u(M)$. Then $PH\cong M$ as $F$-modules. Kuhn showed in \cite[Thm.~1.20]{kuhn:quasi-shuffle} that $H$ is unique among all formal Hopf algebras which are cofree as formal connected coalgebras and split. Here, a formal Hopf algebra $H$ is called split if the inclusion $PH \to \tilde H$ has a retraction as $F$-modules, where $\tilde H$ is the augmentation coideal.

Lemma~\ref{lemma:cofreeiscofreeascoalg} shows that $\Cof^u(A)$ is always cofree as a coalgebra, and by Cor.~\ref{cor:primitiveofcofree}, the composition
\[
A \xrightarrow{\cong} P\Cof^u(A)) \to \widetilde{\Cof^u(A)} \to A
\]
is the identity. This shows that Hopf algebras which are cofree over profinite $p$-polar algebras satisfy Kuhn's condition.
\end{proof}

\section{The cohomology of free iterated loop spaces}\label{sec:iteratedloops}

Throughout this section, let $X$ be a pointed, connected CW-complex and, for simplicity, $k=\F_p$. All homology and cohomology is taken with $\F_p$-coefficients.

Classical iterated loop space theory tells us that
\[
H^*(\Omega^n\Sigma^nX) \cong \bigoplus_{k \geq 0} H^*((C_k)_+ \times_{\Sigma_k} X^{\wedge k}),
\]
where $C_k$ is the $k$th space of the little $n$-disks operad. This splitting is induced from the stable Snaith splitting. It is emphatically \emph{not} a splitting of Hopf algebras (not even of algebras). By \cite[Chapter III]{cohen-lada-may}, $H^*(\Omega^n\Sigma^n X)$ is a functor $D_n$ of the algebra $H^*(X)$ as Hopf algebras. For instance, $D_1(A) = \Cof^{nc}(A)$ is the cofree \emph{non-cocommutative} Hopf algebra on $A$. 

In \cite{kuhn:quasi-shuffle}, Kuhn showed that for connected spaces $X$ of finite type, $H^*(\Omega\Sigma X)$ only depends on the stable homotopy type of $X$ in the sense that if $H^*(X) \cong H^*(Y)$ as $F$-modules then $D_1(H^*(X)) \cong D_1(H^*(Y))$. To each $F$-module $M$ of finite type, he assigns a Hopf algebra $H(M)$ that is cofree with primitives $P(H(M)) \cong M$. It is, however, not true that $H$ is a functor (and therefore that $D_1$ factors through the category of $F$-modules), as the following example shows.

\begin{example}
(This example arose from a discussion with Nick Kuhn.)

Let $M = \F_p\langle x, Fx\rangle$ be the two-dimensional $\F_p$-module generated by a class $x$ in degree $2$ and a class $Fx$ in degree $2p$ such that $F(x)=Fx$ and $F(Fx)=0$.
Then $H(M) = \F_p[x_{(0)},x_{(1)},\dots]/(x_{(i)}^{p^2})$ with comultiplication
\[
\Delta(x_n) = \sum_{i+j=n} x_i \otimes x_j,
\]
where $x_{a_0+a_1p+\cdots+a_np^n} = \frac1{a_0!\cdots a_n!}x_{(0)}^{a_0}x_{(1)}^{a_1}\cdots x_{(n)}^{a_n}$ \cite[Theorem~3.1]{kuhn:quasi-shuffle}. Furthermore, $H(M^l) \cong \prod_{i=1}^l H(M)$, where the product is taken in the category of non-cocommutative Hopf algebras.
If $H$ could be made functorial with $P \circ H \cong \id$, we would have the following sequence of functors:
\[
\Mod_F \xrightarrow{H} \Hopf{\F_p}^{nc} \xrightarrow{\ab} \Hopf{\F_p} \xrightarrow{D} \DMod{\F_p},
\]
where $\ab$ denotes the functor associating to a non-cocommutative Hopf algebra $H \in \Hopf{k}^{nc}$ its maximal cocommutative Hopf algebra. This functor is right adjoint to the inclusion functor and hence sends products of non-cocommutative Hopf algebras to products of cocommutative Hopf algebras, i.e. tensor products.
We have that
\[
D(H(M)) \cong \Bigl\{ 
\begin{tikzcd}\Z/p \arrow[r,bend left=10,"p"] & \Z/p^2 \arrow[l,bend left=10,"1"] \arrow[r,bend left=10,"p"] & \Z/p^2 \arrow[l,bend left=10,"1"] \arrow[r,bend left=10,"p"] & \cdots \arrow[l,bend left=10,"1"] 
\end{tikzcd}
\Bigr\}
\]
where arrows to the left are $V$ and arrows to the right are $F$. 

Note that $\End_{\Mod_F}(M) \cong \Z/p$ and $\End_{\DMod{\F_p}}(D(H(M))) \cong \Z/p^2$, and hence $\Aut_{\Mod_F}(M^l) \cong \GL_l(\Z/p)$ and $\Aut_{\DMod{\F_p}}(D(H(M^l))) \cong \GL_l(\Z/p^2)$. Since $P \circ H \cong \id$, functoriality would imply a section of the mod-$p$ reduction homomorphism $\GL_l(\Z/p^2) \to \GL_l(\Z/p)$. However, by \cite[p.~22]{sah:cohomology-of-split-group-extensions-II}, this map splits exactly when $l=1$ or when $l=2$ and $p \leq 3$ or when $l=3$ and $p=2$. So there is a counterexample for every prime.
\end{example}

Our result gives a functorial factorization through a category that retains a little more structure than just a Frobenius. 

To prove Thm.~\ref{thm:omegansigman}, we need to study the algebraic structure of $H_*(\Omega^{n+1}\Sigma^{n+1} X)$ more closely. For the reader's convenience, we recall the necessary details from \cite[III.1--4]{cohen-lada-may}, with some corrections and additions from \cite{wellington:unstable-adams}. An $(n+1)$-fold loop structure on a space $X$ gives $H=H_*(\Omega^{n+1}\Sigma^{n+1}X)$ the structure of an algebra over the Dyer-Lashof algebra satisfying certain unstability conditions, quite analogously to the structure of cohomology rings as algebras over the Steenrod algebra. However, for finite $n$, there is an additional piece of structure: the so-called Browder operations $[-,-]\colon H_{q} \otimes H_{r} \to H_{q+r+n}$. Additionally, the top Dyer-Lashof $Q^{\frac{q+n}2}$operation is not linear; its nonlinearity is measured by said Browder operations. The Browder operations together with the top Dyer--Lashof operation gives $H_*$ the structure of a ($n$-shifted) restricted Lie algebra. We now outline the algebraic structure on the homology of an $(n+1)$-fold loop space, omitting details about signs to avoid too much clutter. The exact formulas are of little relevance here -- what is mostly important is what part of the algebraic structure depends on what other parts. For simplicity, we assume that $p>2$; the case of $p=2$ is similar but slightly simpler.

\begin{defn}
Let $p=\operatorname{char}(k)>2$ and let $M$ be a graded $k$-vector space. An $R_n$-structure on $M$ consists of:
\begin{enumerate}
	\item Operations $Q^r\colon M_q \to M_{q+2r(p-1)}$ for $0 \leq 2r \leq q+n$;
	\item Operations $\beta Q^r\colon M_q \to M_{q+2r(p-1)-1}$ for $1 \leq 2r \leq q+n$;
	\item A \emph{Browder operation} $[-,-]\colon M_q \otimes M_r \to M_{q+r+n}$
\end{enumerate}
satisfying the following axioms:
\begin{enumerate}[resume]
	\item $[-,-]$ is bilinear;
	\item $[x,y]=\pm [y,x]$;
	\item $[x,[y,z]] \pm [y,[z,x]] \pm [z,[x,y]]=0$;
	\item $[x,Q^r y]=0$ for $2r<q+n$ and $[x,\beta Q^r y]=0$ for $2r\leq q+n$;
	\item $[x,Q^{\frac{q+n}2}y] = [y,[y,\cdots,[y,x]\cdots]$ ($p$-fold bracket) for $q+n$ even;
	\item $Q^rx=0$ if $2r<|x|$; $\beta Q^rx=0$ if $2r\leq |x|$;
	\item All $Q^r$ and $\beta Q^r$ are additive except for $Q^{\frac{q+n}2}$ when $q+n$ is even;
	\item $Q^r(\lambda x) = \lambda^p Q^r(x)$; $\beta Q^r(\lambda x) = \lambda^p \beta Q^r(x)$
	\item $Q^{\frac{q+n}2}(x+y) = Q^{\frac{q+n}2}(x)+Q^{\frac{q+n}2}(y)+(x^p+y^p-(x+y)^p)$ when $q+n$ is even; here the parenthesized summand is understood as evaluated in the universal enveloping (noncommutative) algebra of the Lie bracket $[-,-]$ and can be shown to lie in $M$;
	\item $\beta^\epsilon Q^rQ^s = \sum_{i> 0} \pm \binom{(p-1)(r-i)-1}{i-ps-1}
	 \beta^\epsilon Q^iQ^{r+s-i}$ for $r > ps$, $\epsilon \in \{0,1\}$;
	\item $Q^r\beta Q^s = \sum_{i> 0} \pm \Bigl(\binom{(p-1)(r-i)}{i-ps} \beta Q^{i} Q^{r+s-i} \pm \binom{(p-1)(r-i)-1}{i-ps} Q^{i}\beta Q^{r+s-i}\Bigr)$ for $r \geq ps$;
	\item $\beta Q^r\beta Q^s = \sum_{i>0} \pm \binom{(p-1)(r-i)-1}{i-ps} \beta Q^{i}\beta Q^{r+s-i}$ for $r \geq ps$.
\end{enumerate}
An \emph{$R_n$-algebra} is an $R_n$-module $M$ with a commutative, unital multiplication satisfying
\begin{enumerate}[resume]
	\item $[x,1]=0$;
	\item $[x,yz] = [x,y]z \pm y[x,z]$;
	\item $Q^r1 = 0$ and $\beta Q^r 1 = 0$ for $r>0$;
	\item $Q^rx = x^p$ if $2r=|x|$;
	\item $Q^r(xy) = \sum_{i+j=r} Q^i(x)Q^j(y)$ for $2r<q+n$;
	\item $\beta Q^r(xy) = \sum_{i+j=r} (\beta Q^i(x)Q^j(y)+ Q^i(x)\beta Q^j(y))$ ;
	\item $Q^{\frac{q+n}2}(xy) = \sum_{i+j=\frac{q+n}2}Q^{i}(x)Q^{j}(y)+\Gamma(x,y)$ when $q+n$ is even, where $\Gamma$ is a certain function of $x$ and $y$ constructed from the multiplication and Browder bracket in $M$;
\end{enumerate}
An \emph{$R_n$-Hopf algebra} is an $R_n$-algebra $M$ with a cocommutative comultiplication $\Delta$ making $M$ into a Hopf algebra and satisfying
\begin{enumerate}[resume]
	\item $\Delta Q^r(x) = Q^r(\Delta(x))$, where $Q^r\colon M \otimes M \to M \otimes M$ is given by the external Cartan formula: $Q^r(x \otimes y) = \sum_{i+j=r}  Q^i(x) \otimes Q^j(y)$; \label{ax:diagonalexternalcartan}
	\item $\Delta \beta Q^r(x) = \sum_{i+j=r} \sum_{(x)} (\beta Q^i(x') \otimes Q^j(x'')+Q^i(x') \otimes \beta Q^j(x''))$ \label{ax:diagonalexternalbetacartan}
	\item $\Delta [x,y] = \sum_{(x)} \sum_{(y)} \Bigl(\pm [x',y'] \otimes y''x'' \pm  x'y' \otimes [x'',y'']\Bigr)$.
\end{enumerate}
\end{defn}
(Axiom~\eqref{ax:diagonalexternalcartan} seems to be formulated in a convoluted way, but the formula analogous to \eqref{ax:diagonalexternalbetacartan} would be wrong when $r=\frac{q+n}2$ because of the nonlinearity of $Q^r$.)

Denote by $\Hopf{R_n}$ the category of $R_n$-Hopf algebras thus equipped. There is a forgetful functor
\begin{equation}
\Hopf{R_n} \xrightarrow{U} \Hopf{k} \to \Coalg_k, \label{eq:forgetfulDyerLashof}
\end{equation}

It is shown in \cite[Thm.~3.2]{cohen-lada-may} that this composite has a left adjoint $W$ such that
\[
W(H_*(X)) \cong H_*(\Omega^{n+1}\Sigma^{n+1} X) \quad \text{for $X$ connected, $n \geq 1$.}
\]

However, the left adjoint is not constructed as a composition of adjoint functors according to \eqref{eq:forgetfulDyerLashof}. Thus we need to show:
\begin{lemma}\label{lemma:freeRnHopfalgebra}
The functor $U\colon \Hopf{R_n} \to \Hopf{k}$ has a left adjoint $F$.
\end{lemma}
\begin{proof}
While it is possible to give an explicit construction, we give a proof using Freyd's adjoint functor theorem.

The category $\Hopf{k}$ has all small limits and we will show that $U$ creates (and thus preserves) limits in $\Hopf{R_n}$. Given two $R_n$-Hopf algebras $H_1$ and $H_2$, the usual tensor product $H_1 \otimes_k H_2$, which is the biproduct in $\Hopf{k}$, becomes an $R_n$-Hopf algebra by the external Cartan formulas for $Q^i$, $\beta Q^i$, and the Browder bracket. Given two maps $f,g\colon H_1 \to H_2$ of $R_n$-Hopf algebras, their equalizer (as graded sets) is a $k$-Hopf algebra and also the equalizer in $\Hopf{k}$. In fact, it is an $R_n$-Hopf algebra because $f$ and $g$ commute with the $R_n$-structure. Since $U$ is conservative (reflects isomorphisms), it thus creates finite limits.
Finally, given an inverse system $H\colon I \to \Hopf{R_n}$, its limit in $\Hopf{k}$ is given as
\[
\lim_{I} (U \circ H) = \{ x \in \lim H \mid \Delta(x) \in \lim H \otimes \lim H\}, 
\]
where the limits on the right hand side are limits of graded sets (equivalently, graded abelian group or graded commutative algebras). The diagonal Cartan formulas show that this limit is closed under $Q^r$, $\beta Q^r$, and the Browder bracket, so that $\lim_I H$ is created by $\lim_I U \circ H$.

Furthermore, $U$ commutes with filtered colimits, which are also created in graded sets.

To apply Freyd's special adjoint functor theorem, we show that $\Hopf{k}$ and $\Hopf{R_n}$ are locally presentable categories. This is well-known for $\Hopf{k}$ (it is, after all, equivalent to Dieudonné modules). Since $\Coalg_k$ is locally presentable and $\Hopf{R_n}$ is monadic over it (either by the explicit form of Cohen's construction of the free $R_n$-Hopf algebra on a coalgebra or an application of Beck's monadicty theorem), it is also locally presentable \cite[2.78]{adamek-rosicky:presentable-accessible}.
\end{proof}

\begin{proof}[Proof of Thm.~\ref{thm:omegansigman}]
By \cite[Thm.~3.2]{cohen-lada-may},
\[
H_*(\Omega^{n+1}\Sigma^{n+1} X) \cong W(H_*(X)) \underset{\text{Lemma~\ref{lemma:freeRnHopfalgebra}}}\cong F(\Fr(H_*(X))).
\]
Dually,
\[
H^*(\Omega^{n+1}\Sigma^{n+1} X) \cong F^*(\Cof(H^*(X))),
\]
where $F^*$ is the functor $F$ on opposite categories. By the formal case of Thm.~\ref{thm:freeschemefactorization}, the desired factorization is given by
\[
A \mapsto F^*(\Cof(A)) \quad \text{for $A \in \pol_p(k)$.}
\]
\end{proof}

We obtain the following generalization of Kuhn's result, a strengthening of Cor.~\ref{cor:kuhnthm} from the introduction:

\begin{corollary}
For $n \geq 0$, and spaces $X$ for finite type, the $R_n$-Hopf algebra $H^*(\Omega^{n+1}\Sigma^{n+1}X)$ only depends on the $F$-module $H^*(X)$ (up to noncanonical isomorphism), and in particular only on the stable homotopy type of $X$.
\end{corollary}
\begin{proof}
The case $n=0$ was proved in \cite[Cor.~7.2]{kuhn:quasi-shuffle}. For $n \geq 1$, \cite[Thm.~1.12]{kuhn:quasi-shuffle} shows that $\Cof^{nc}(A)$ only depends on the structure of $A$ as an $F$-module, up to noncanonical isomorphism. Here $\Cof^{nc}(A)$ is the cofree non-cocommutative formal Hopf algebra on $A$, but the cocommutative case follows since the inclusion of cocommutative into not necessarily cocommutative Hopf algebras has a right adjoint (cf. \cite[Remark 3.13]{kuhn:quasi-shuffle}). Since $H^*(\Omega^{n+1}\Sigma^{n+1}) = F^*(\Cof(H^*X))$, the result follows.  
\end{proof}

\bibliographystyle{alpha}
\bibliography{bibliography}

\end{document}